\documentclass{article}

\usepackage{graphicx}
\usepackage[fleqn]{amsmath}
\usepackage{amsthm}
\usepackage{amssymb}
\usepackage{amsbsy}
\usepackage{amsfonts}
\usepackage{mathrsfs}
\usepackage[all]{xy}
\usepackage{amstext}
\usepackage{amscd}
\usepackage[dvips]{epsfig}
\usepackage{psfrag}
\usepackage{enumerate}
\usepackage{flafter}
\allowdisplaybreaks

\theoremstyle{plain}
\newtheorem{thm}{Theorem}[section]
\newtheorem{prop}[thm]{Proposition}
\newtheorem{cor}[thm]{Corollary}
\newtheorem{lem}[thm]{Lemma}
\theoremstyle{definition}
\newtheorem{exa}[thm]{Example}

\newtheorem{rem}[thm]{Remark}
\newtheorem{defn}[thm]{Definition}

\def\dim{\mathop{\mathrm{dim}}\nolimits}

\def\Hom{\mathop{\mathrm{Hom}}\nolimits}

\newcommand{\lra}{\longrightarrow}
\newcommand{\ra}{\rightarrow}
\newcommand{\Q}{{\Bbb Q}}
\newcommand{\R}{{\Bbb R}}
\newcommand{\Z}{{\Bbb Z}}
\newcommand{\N}{{\Bbb N}}

\newcommand{\A}{{\cal A }}

\newcommand{\id}{{\mathrm{id}}}

\newcommand{\C}{\mathbb{C}}

\newcommand{\pc}[2]{\mbox{$\begin{array}{c}
\includegraphics[scale=#2]{#1.eps}
\end{array}$}}

\begin{document}
\large
\begin{center}{\bf\LARGE de Rham theory and cocycles of cubical sets from smooth quandles}
\end{center}

\vskip 1.5pc

\begin{center}{Takefumi Nosaka\footnote{
E-mail address: {\tt nosaka@math.titech.ac.jp}
}}\end{center}

\vskip 1pc

\begin{abstract}
We show a de Rham theory for cubical manifolds, and study rational homotopy type of the classifying spaces of smooth quandles.
We also show that secondary characteristic classes in \cite{Dup2,DK} produce
cocycles of quandles.
\end{abstract}

\begin{center}
\end{center}

\large
\baselineskip=16pt
\section{Introduction}\label{Adem}
Characteristic classes in topology are interpreted as cohomology classes of the classifying space of a Lie group $G$. According to Chern-Weil theory, the classes recovered from some invariant theory. Dupont \cite{Dup} used simplicial manifolds to study the classifying spaces, and reformulate universally the Chern-Weil theory.
Moreover, according to the enriched Chern-Weil theory \cite{Dup2,DK},
the characteristic classes (with a condition) produce
cocycles of $G^{\delta} $, where $G^{\delta}$ is the Lie group $G$ with descrete topology.
This approach recovers some of secondary characteristic classes, including the Chern-Simon class. 

Meanwhile,
a quandle \cite{Joy,Mat} is a set with a certain binary operation;
a typical example is 
a homogenous set as in symmetric space (see \S\S \ref{das}--\ref{a2das23} for the details).
Furthermore, as an analog of the classifying space of a group,
Fenn, Rourke, and Sanderson \cite{FRS1} defined a space $BX$ from a quandle $X$,
which is called the rack space, 
and is cubically constructed from a $\square$-set; cocycles in the cohomology provided applications to low-dimensional topology (see \cite{CJKLS,CKS}), e.g., including the Chern-Simon invariant \cite{IK} and $K_2$-invariant \cite{Nos3} of links.
However, in most papers on quandles, $X$ was assumed to be equipped with descrete topology.

In this paper, we focus on the situation where a quandle $X$ has a manifold structure as a homogenous space, and we study the cohomology of $BX$.
After Section \ref{das} reviews quandles with manifold structure, Section \ref{a2das23} discusses differential forms on cubical manifolds, and shows a de Rham theory on $BX$ (Theorem \ref{thm2}): This result is a cubical translation of \cite{Dup}. As a corollary, Section \ref{a2das} completely determines the rational cohomology of the rack
space $BX$, where the cohomology of $X$ satisfies some conditions.
Furthermore, for such an $X,$ Section \ref{r4das} provides a formula of computing the rational homotopy type of $BX$, as in Milnor-Moore theorem; see Theorem \ref{prp:}.

In Sections \ref{A33}--\ref{A326}, we will examine a contrast between the cohomology groups of $BX$ and $BX^{\delta}$, where $X^{\delta}$ means the discrete topology of $X$. 
First, we show (Theorem \ref{co43e}) that if $X$ is compact and ``semi-homogenous", every $\R$-value continuous cocycle of $BX^{\delta}$ is trivial (cf. the computation of second (co)-homology of $BX^{\delta}$; see Appendix \ref{B1}). 
To obtain non-trivial cocycles, the last section \ref{A326} examines cocycles with the coefficient $\mathbb{C}/\Z$ modulo $\Z$, where we use a chain map of Inoue-Kabaya \cite{IK} to bridge 
the complex of $BX^{\delta}$ and the enriched Chern-Weil theory.
As a result, we show (Proposition \ref{d4e5}) that
the pullback of every secondary characteristic class in the sense of \cite{Dup2,DK} yields a $\C/\Z$-value cocycle of $BX^{\delta}$. Hence, in doing so, we hope that this proposition produces many rack cocycles of a quandle $Y$, when $Y$ is a subquandle of $X$.

\subsection*{Acknowledgment}
The author sincerely expresses his gratitude to Katsumi Ishikawa and Masahico Saito
for valuable comments on the early draft of this paper. He also thanks Hiroshi Tamaru for referring him to the papers \cite{Nagano,NT}.

\section{Preliminaries on smooth quandles}\label{das}
We start reviewing quandles and smooth quandles.
A {\it quandle} \cite{Joy,Mat} is a set $Q$ with a binary operation $\lhd :Q^2 \ra Q$ satisfying the following three:

\vskip 0.7pc
\noindent
(Q1)
For any $x \in Q, x \lhd x = x$,

\noindent
(Q2) For any $x, y \in Q$, there exists a unique element $z \in Q$ such that $z\lhd y = x$,

\noindent(Q3) For any $x, y, z \in Q$, $(x \lhd y)\lhd z = (x\lhd z)\lhd (y\lhd z)$.

\vskip 0.7pc
A {\it smooth quandle} is a $C^{\infty}$-manifold $Q$ with a $C^{\infty}$-map $\lhd : Q^2 \ra Q$
satisfying (Q1), (Q3) and that $ (\bullet \lhd x) : Q \ra Q$ is diffeomorphic for any $x \in Q$.
Let $\mathrm{Inn}(Q)$ be the subgroup of $\mathrm{Diff}(Q)$ generated by $ (\bullet \lhd y)$, where $y$ runs over $Q$.
We equip $\mathrm{Inn}(Q) \subset \mathrm{Diff}(Q)$ with the compact open topology.
A quandle $Q$ is said to be {\it transitive}, if the action of $\mathrm{Inn}(Q)$ on $Q$ is transitive; see \cite{Joy,Mat}. 
A quandle $Q$ is {\it of type} $n$, if there exists $n \in \Z $ which is the minimal number satisfying $x \lhd^n y=x$ for any $x, y \in Q$.

\begin{exa}\label{kkk3}Let $X$ be a symmetric space, i.e., a $C^{\infty}$-manifold equipped with a Riemannian metric such that each point $y \in X$ admits an isometry
$s_y : X \ra X$ that reverses every geodesic line $\gamma : (\mathbb{R}, 0) \ra (X,y)$, meaning that
$s_y \circ \gamma (t) = \gamma ( - t)$.
\index{symmetric space@symmetric space}
Then, $X$ has a quandle structure of type $2$ defined by $x \lhd y := s_y (x)$.
In addition, similar Riemannian manifolds with quandle structure of type $>2$ are studied in \cite{Kow} as {\it generalized symmetric spaces}. 
\end{exa}
\begin{exa}[\cite{Joy,Mat}]\label{kkk}
As an important example in this paper, we will see that transitive quandle structures turn to be good operations defined on homogenous
spaces.
Let $G$ be a Lie group, and $H$ be a closed subgroup.
If $z_0 \in G$ commutes with any $h\in H$,
then the homogenous space $H \backslash G$ has a quandle structure given by
\begin{equation}\label{qkettei} [x] \lhd[y] := [z_0^{-1} x y^{-1} z_0 y],
\end{equation}
for representatives $x,y \in G$. In what follows, we write $(G,H,z_0)$ for such a transitive quandle.
The author should keep the map $\kappa: H \backslash G \ra G$ which sends $[x]$ to $ x^{-1}z_0 x$ in mind.

Conversely, we will explain that if $Q$ is a smooth quandle and is transitive, $Q$ is reduced to some $(G,H,z_0)$.
For $x_0 \in Q$, let $\mathrm{Stab}(x_0)\subset G$ be the stabilizer subgroup of $x_0$.
We equip the group $\mathrm{Inn} (Q)$ with a quandle operation given by (\ref{qkettei}).
Then it is known \cite[Theorem 7.1]{Joy} that the natural map
\begin{equation}\label{gxgx} \mathrm{Inn}(Q) \lra Q \ \ \ \ \ \ \mathrm{given \ by} \ \ \ g \longmapsto x_0 \cdot g
\end{equation}
is a quandle homomorphism, which induces the quandle isomorphism $ \mathrm{Stab}(x_0) \backslash \mathrm{Inn}(Q) \cong Q$.
Moreover, Ishikawa \cite[Theorem 2.4]{Ishi} showed that $\mathrm{Inn}(Q) $ is a Lie group.
In conclusion, the structure of the smooth quandle $Q$ is determined by the Lie groups $\mathrm{Stab}(x_0) \subset \mathrm{Inn}(Q) $.

\end{exa}

Accordingly, throughout this paper, we mainly focus on such smooth quandles $ (G,H,z_0)$, which are transitive quandles.


Moreover, we now observe the situation that $G $ is compact.
Then $G$ has the Haar measure $dg$.
By quotienting $dg$, the smooth quandle $Q$ has a metric such that
$ (\bullet \lhd x): Q\ra Q$ is isometric for any $x \in Q$.
In other words, such a smooth quandle $Q$ is called a metrizable $s$-manifolds in the book \cite{Kow}.
Hence, the topological type of such a $Q$ is restricted, and is classified in some cases.
For example, if $\pi_1(Q)=0$, the type is of finite order, and $G$ is a simple Lie group, then $Q$ is a formal space in the sense of the rational homotopy theory; see \cite{KT} and references therein.

\section{Preliminaries: cubical manifolds and differential $n$-forms}\label{a2das23}

We introduce cubical manifolds, modifying the concept of $\square$-sets of Fenn-Rourke-Sanderson \cite{FRS1}. 
The discussion in this section is a cubical analogy of simplicial manifolds \cite[\S 2]{Dup}.
A {\it cubical manifold} is a sequence of
$C^{\infty}$-manifolds $\{ X_p \}_{p \in \N} $ together with {\it face} $C^{\infty}$-{\it maps} \ $\delta^{\varepsilon}_{i}: X_{p} \ra X_{p-1}$, for $ \varepsilon \in \{ 0,1\}$ and $1 \leq i \leq p$, satisfying
\[ \delta^{\eta}_{j-1} \circ \delta^{\varepsilon}_{i} = \delta^{\varepsilon}_{i} \circ \delta^{\eta}_{j }, \ \ \ \mathrm{for \ any} \ 1 \leq i <j \leq p\ \ \mathrm{and} \ \varepsilon , \eta \in \{ 0,1\}.\]
Let $I$ be the interval $[0,1] \subset \R$, and $I^p$ be the $p$-cube.
Dually, for $ 1 \leq i \leq p$ and $ \varepsilon \in \{ 0,1\}$,
we consider the map
$$\delta^{\varepsilon}_{i}: I^{p-1} \ra I^{p} \ \ \mathrm{ defined \ by }\ \ \delta^{\varepsilon}_{i} (t_1, \dots, t_{p-1})= (t_1,\dots, t_{i-1} , \varepsilon, t_{i}, \dots, t_{p-1} ).$$
Then the ({\it fat}) {\it realization} $\|X \|$ of a cubical manifold $X$ is defined to be the quotient space of
$ \bigsqcup_p I^p \times X_p $ subject to the relation $( \delta^{\varepsilon}_{i}(\mathfrak{t}) ,x) \sim (\mathfrak{t}, \delta^{\varepsilon}_{i}(x)),$
where $ \mathfrak{t} \in I^{p-1} $ and $x \in X_p$ with $i= \{ 0, \dots, p\} $ and $ \varepsilon \in \{ 0,1\}$.
\begin{exa}[Rack space]\label{rs}
Fenn-Rourke-Sanderson \cite{FRS1} introduced a classifying space as a cubical set, which is called {\it the rack space}. We will give the rack space of manifold version.
Fix a smooth quandle $(G,H,z_0)$ as in Example \ref{kkk}, and a manifold $Y$ which is acted on by $G$ (possibly $Y=\{ \mathrm{pt}.\}$, $Y=Q$ or $ Y=G$).
Then, we define $X_p$ to be $Y \times Q^p$, and define $ \delta_{j}^{\varepsilon}$ by
\[ \delta_{j}^0(y,x_1, \dots, x_p) = (y,x_1, \dots, x_{j-1}, x_{j+1},\dots, x_p), \]
\[ \delta_{j}^1(y,x_1, \dots, x_p) = (y \cdot \kappa (x_j) ,x_1 \lhd x_j, \dots, x_{j-1} \lhd x_j, x_{j+1}, \dots, x_p). \]
Then, the pair ($X_*, \delta_{*}^{\varepsilon}$) is a cubical manifold. Moreover, the realization $| \! |X| \! |$ is exactly the rack space defined in \cite{FRS1,FRS2}. We will denote $| \! |X| \! |$ by $B_Y Q$. If $Y$ is a singleton, we write $BQ$ for $B_YQ$ for simplicity. We remark that the canonical projection $B_YQ \ra BQ$ is a fibration with fiber $Y.$
\end{exa}
Next, we will establish terminology of $C^{\infty}$-forms on cubical manifolds. 
\begin{defn}\label{form}
(1) Let $\A ^n( I^p \times X_p)$ be the set of $n$-forms on $I^p \times X_p$ of $C^{\infty}$-class which are extended to $n$-forms on $\R^p \times X_p $.

\noindent
(2) An {\it $n$-form} $\varphi$ on a cubical manifold is
a sequence of $n$-forms $\phi^{(p)} \in \A ^n(I^p \times X_p)$ satisfying the conditions
$ (\delta^{\varepsilon}_{i} \times \id)^* \phi^{(p)}= (\id \times \delta^{\varepsilon}_{i} )^* \phi^{(p-1)} $ for any $i\in \{1, \dots, p\}$ and $ \varepsilon \in \{ 0,1\}$.

\noindent
(3) We denote by $ A^n (X)$ the set of all $n$-forms on $X$.
\end{defn}
\noindent
Then, the exterior differential $d$ and the wedge product on $ \A ^n( I^p \times X_p)$ can be extended to those on $A^n (X)$. Thus, $ A^* (X)$ is made into a differential graded algebra.

Next, we give bigraded complexes. 
Let $q_{1} : I^p \times X_p \ra I^p $ and
$q_{2} : I^p \times X_p \ra X_p$ be the natural projections.
Given a cubical manifold $X$, we first decompose $ A ^* (X)$ into a direct sum $A ^n (X)= \bigoplus_{n=k+\ell } A^{k,\ell}(X)$,
where $ A^{k,\ell }(X)$ is composed of the forms $\varphi$ of type $(k,\ell )$,
i.e., $\varphi$ restricted to $ I^p \times X_p$ is presented by $ q_1^* (\phi^{(k)}_I ) \times q_2^* (\phi^{(\ell)}_X ) $
for some $ \phi^{(k)}_I \in \A^k (I^p)$ and $ \phi^{(\ell)}_X \in \A^{\ell} (X_p)$.
Also let $d_{\square}$ (resp. $d_X$) denote the pullback of exterior differential on $\A^* (I^p)$ (resp. on $\A^* (X_p) $).
Thus, we have a double complex $(A^{k,\ell}(X), d_{\square}, d_X )$,
and the total complex $(A^*(X), d_{\rm tot})$, where $d_{\rm tot}= d_{\square}+ d_X $.
Further, we can define another double complex $(\mathcal{A}^{k,\ell}(X), \delta, d_X )$, where $\mathcal{A}^{k,\ell}(X) = \A^\ell(X_k) $ and $\delta= \sum_{i=1}^p (-1)^i (\delta_i^0 -\delta_i^1 ) .$

Then, we later give an isomorphism between the (double) complexes
\begin{thm}[{A cubical version of \cite[Theorem 2.3]{Dup}}]\label{thddm1}
Assume that each $X_p$ is a paracompact Hausdorff space.
For any $\ell \in \N $ the chain complexes $(A^{*,\ell}(X), d_{\square})$ and
$(\mathcal{A}^{*,\ell}(X), \delta)$ are naturally chain homotopy equivalent.
To be precise, there is a map $\mathcal{J}: A^{k,\ell}(X) \ra \mathcal{A}^{k,\ell }(X)$ which gives a homotopy equivalence.
\end{thm}

Instead of giving the proof (see Appendix \ref{A1}), we mention a corollary from the spectral sequences
associated with the two double complexes.
Consider the filtering with respect to the first index of the double complexes $A^{ **}(X)$ and $\mathcal{A}^{**}(X)$;
we have the spectral sequences $I(A)_r^{**}$ and $I(\mathcal{A})_r^{**}$, respectively. 
In parallel, we have other spectral sequences $II(A)_r^{**}$ and $II(\mathcal{A})_r^{**}$ by filtering with respect to the second index.
As a consequence of Theorem \ref{thddm1}, as a de Rham theory of cubical sets,
the de Rham cohomology of $A^*(X)$ is isomorphic to the ordinary cohomology $H^*(| \! |X| \! |;\mathbb{R})$ of the fat realization $| \! |X| \! |$. To be precise,
\begin{cor}
The map $\mathcal{J}$ induces natural isomorphisms $I(A)_r^{**} \cong I(\mathcal{A})_r^{**}$ for $r\geq 2 $ and $II(A)_r^{**}\cong II(\mathcal{A})_r^{**}$ for $r \geq 1$.
In particular, they induce a canonical isomorphism from the cohomology of the total complexes $\mathcal{K}_X: H^*(A^*(X),d_{\rm tot}) \cong H^*(| \! |X| \! |;\R)$.
\end{cor}
\noindent
Moreover, we will show the multiplication, although we defer the proof into Appendix \ref{A1}.
\begin{thm}[{Cubical version of \cite[Theorem 2.14]{Dup}}]\label{thm2}Suppose that each $X_p $ is a paracompact Hausdorff space. Then the isomorphism $\mathcal{K}_X: H^*(A^*(X),d) \cong H^*(| \! |X| \! |;\R)$ is multiplicative where the multiplication on the left (resp. right) hand side is induced by the wedge-product (resp. the cup-product).
\end{thm}

\section{Note on rational cohomology of the rack spaces }\label{a2das}
In this section, we will compute the rational cohomology of the rack space $BQ$. 
For this, we consider the invariant part, $A^n(Q)^G$, of $n$-forms,
where the action of $G$ on $Q$ is induced from the right actions in \eqref{gxgx}.
We have the inclusion $A^n(Q)^G \hookrightarrow A^n(Q)$.

\begin{prop}\label{cohomology}
Let $Q$ be a smooth quandle of the form $(G,H,z_{0})$. Assume that the inclusion $A^n(Q)^G \hookrightarrow A^n(Q)$ yields an isomorphism on cohomology.
Then, there are isomorphisms
$$ H^n(BQ; \mathbb{R}) \cong \bigoplus_{n=i+j } H^i(Q^{j}; \mathbb{R}) , \ \ \ \ H^n(B_G Q; \mathbb{R}) \cong \bigoplus_{n=i+j } H^i(G \times Q^{j}; \mathbb{R}) .$$
\end{prop}
\begin{proof}
We consider the spectral sequence $II(\mathcal{A})_r^{**}$ in \S \ref{a2das23},
which strongly converges to $E_{\infty}^n \cong H^n (A^{* }(B_G Q)) \cong H^n(B_G Q;\R).$

We will study the $E_{1}^{p,q}$-term $H^p( \A^* (Q^q ) ) $ in detail.
We let $ \A^*(Q^q )^{G^q}$ be the set of $G^q$-invariant forms on $Q^q$,
where $G^q$ acts $Q^q$ componentwise.
By assumption, the inclusion $\A^*(Q^q )^{G^q} \hookrightarrow \A^*(Q^q )$ is a quasi-isomorphism for any $q$. 
For any $G^q$ invariant $p$-form $\psi \in \A^p(Q^q )^{G^q} $, we note
$(\delta_i^0 -\delta_i^1 )^* (\psi )=0$ by definition; therefore, $\delta^*(\psi) =( \sum_{i=1}^q (-1)^i (\delta_i^0 -\delta_i^1 ) )^* (\psi )=0$.
Thus, this spectral sequence collapses at $ E_{2}^{p,q} $, i.e., $E_{2}= E_{\infty }$. Hence, we can get the conclusion:
\[ H^n(B Q;\R) \cong H(A^{n}(B Q)) \cong E_{\infty}^n \cong \bigoplus_{n=i+j } E_{2}^{i,j} \cong \bigoplus_{n=i+j } H^i(Q^{j}; \mathbb{R}).\]
Next, we will show the second isomorphism in a similar way.
Consider the spectral sequence $II(\mathcal{A})_r^{**}$ in \S \ref{a2das23}, where $X_p= G\times X^p.$
Then, we can readily see that this spectral sequence $ E_{2}^{p,q} $ abuts to $ E_{\infty }^{p,q}$.
To conclude, we have the second claim as follows:
\[ H^n(B_G Q;\R) \cong H^n (A^{* }(B_G Q)) \cong E_{\infty}^n \cong \bigoplus_{n=i+j } E_{2}^{i,j} \cong \bigoplus_{n=i+j } H^i(G \times Q^{j}; \mathbb{R}).\]
\end{proof}
Although the assumption in this proposition seems strong, there are many examples.
\begin{exa}\label{rs2x}
If $Q$ is the $2m$-sphere, and $G$ is the orthogonal group $O(2m)$, then the generator of $H^{2m}(S^{2m}) \cong \R$ is represented by the $O(2m)$-invariant volume form. 
Thus, $A^*(Q)^G \hookrightarrow A^*(Q)$ is quasi-isomorphic. 

As another example, consider the unitary group $G=U(m)$ and the Grassmann manifold $\mathrm{Gr}(m,n)$ over $\C$, where $m,n \in \mathbb{N}$ with $n <m$. The cohomology is generated by the Chern classes. Chern-Weil theory implies that the Chern classes is invariant with respect to the action of $U(m)$. Hence, this situation satisfies the assumption.

In general, if $G$ is compact, the Cartan algebra of $G/H$ enables us to compute $H^n(G/H; \mathbb{R}) $ with generators from some information of $\bigwedge^* \mathfrak{g}$, where $\mathfrak{g} $ is the Lie algebra of $G$; see \cite{Ra} and references therein for the details. Thus, we can check whether $G/H$ satisfies the assumption or not.
\end{exa}
\begin{rem}
As seen in the proof, for $Q= (G,H,z_0)$, the inclusion $A^n(Q)^G \hookrightarrow A^n(Q)$ gives rise to a ring homomorphism $ H^*(BQ; \R)\ra \bigoplus_{n=i+j } H^i(Q^{j}; \mathbb{R}) $. However, in general, it seems far from an isomorphism.

For example, if $Q=S^{2n-1}$ and $G=O(2n-1)$, $Q$ does not satisfy the assumption.
Moreover, as a private communication, Ishikawa pointed out that the cohomology of $BQ$ is far from the result of Proposition \ref{cohomology}.
\end{rem}

We give an example of computing $H_*(BQ)$ where $Q$ is the $2m$-sphere:
\begin{exa}\label{rsx}
Let $Q$ be the $2m$-sphere, $S^{2m}$, as a symmetric space, i.e., a quandle of type 2.
Then, $H^k_{dR}(Q) \cong \R $ if and only if $k=0$ and $k=2m. $
Therefore, for $k ,j \geq 0$,
the dimension of $ H^{2mj } (Q^k )$ is equal to $ \tbinom{k}{j}$.
Hence, the Poincar\'{e} series $\sum_{k} \mathrm{dim} H^{k} (BQ;\R )s^k $ is
\begin{equation}\label{g999} \sum_{j=0}^{\infty} \sum_{k=0}^{\infty} s^{ 2mj } \tbinom{k}{j} s^k =\sum_{k=0}^{\infty} \sum_{j=0}^k s^{ 2mj +k } \tbinom{k}{j}= \sum_{k=0}^{\infty} (1+s^{2mk})s^k = \frac{1}{1-s -s^{2m+1}} \in \Z[\! [ s ] \! ]. \end{equation}
\end{exa}

\section{Rational homotopy group of the rack spaces}\label{r4das}
We will show Theorem \ref{prp:} of computing the rational homology of $BQ$,
\begin{thm}\label{rqhbgq}\label{prp:}
Let $Q$ be a smooth quandle of the form $(G,H,z_{0})$. Suppose that $G$ is connected and compact, and satisfies the same assumption in Proposition \ref{cohomology}.
Let $u_i= \dim \pi_i(B Q) \otimes \Q$. Then, the following equality holds:
\begin{equation}\label{g99} \sum_{k \geq 0 } \mathrm{dim}(H^k (B Q; \mathbb{R}) ) s^k= \prod_{i=0}^{\infty} \frac{(1+s^{2i+1})^{u_{2i+1}}}{ (1-s^{2i})^{u_{2i}}} \in \Z [\! [s ] \!] . \end{equation}
\end{thm}
\begin{rem} The homotopy group $\pi_i(B Q)$ contains $\pi_* (\Omega S^2 )$ as a direct summand.
Indeed, letting $P$ be the quandle on the single point, any maps $Q \ra P$ and $P \ra Q$ are quandle homomorphisms, and $ B P \simeq \Omega S^2$ is shown \cite{FRS2,FRS3}.
\end{rem}

To prove the theorem, we review a monoid structure on $B_GQ$, following \cite{Cla}.
For any $n,m \in \N$, we take a map $\mu: ( I^n \times G \times Q^n) \times ( I^m \times G \times Q^m ) \ra I^{n+m} \times G \times Q^{n+m}$ defined by
\[\mu ([t_1, \dots, t_n, g,x_1 \dots, x_n], [ t'_1, \dots, t'_m ,h,x'_1 \dots, x'_m] )\]
\begin{equation} \ \ \ \ \ := [ t_1, \dots, t_n, t'_1, \dots, t'_m, gh, x_1h ,\dots x_n h, x'_1 \dots, x'_m ].
\notag \end{equation}
\noindent
Regarding $B_G Q$ as a quotient of
$ \bigsqcup_p (I^p \times G \times Q^p) $, this $\mu$ passes to a binary operation $B_G Q \times B_G Q \ra B_G Q$, which
makes $B_G Q$ into an associative topological monoid with unit \cite[\S 2.5]{Cla}.
Recall a well-known fact that there exists a simplicial set $Z$ such that $B_G Q$ is weak equivalent to a (based) loop space $\Omega Z$ as an $H$-space. 

Next, we will observe the equality \eqref{g09} below from Milnor-Moore theorem.
Here, since $Q$ and $G$ are compact, $B_GQ$ is a CW-complex of finite type; hence, so is $Z$ (see \cite{FHT} for more detail).
Since the space $B_G Q$ is connected by assumption, we notice $ \pi_0(Z) \cong 0$ and $ \pi_1(Z) \cong \pi_0(B_G Q) \cong 0$, that is,
the space $Z$ is simply connected.
Since the cohomology group $H^*(B_G Q; \mathbb{R})$ is made into a Hopf algebra, Milnor-Moore theorem (see \cite[\S 21]{FHT}) immediately implies the isomorphisms
$$ \mathrm{Prim}( H^*(B_G Q; \mathbb{Q})) \cong \mathrm{Prim}( H^*(\Omega Z ; \mathbb{Q})) \cong \pi_* ( \Omega Z ) \otimes \Q\cong \pi_* ( B_G Q ) \otimes \Q,$$
where $\mathrm{Prim} ( H^*(B_G Q; \mathbb{Q} ))$ means the subspace consisting of primitive elements of $H^*(B_G Q; \mathbb{Q})$.
Then, the Poincar$\acute{\mathrm{e}}$-Birkoff-Witt theorem (see \cite[\S 33(c)]{FHT}) directly leads to
\begin{equation}\label{g09} \sum_{k \geq 0 } \mathrm{dim}(H^k (B_G Q; \mathbb{R}) ) s^k= \prod_{i=0}^{\infty} \frac{(1+s^{2i+1})^{r_{2i+1}}}{ (1-s^{2i})^{r_{2i}}} \in \Z [\! [s ] \! ],
\end{equation}
where $r_i= \dim \pi_i(B_G Q) \otimes \Q$.

\begin{proof}[Proof of Theorem \ref{prp:}]
First, notice that the natural projection $B_G Q \ra BQ $ is
a principal (topological) $G$-bundle (see \cite[\S 3]{FRS1} or \cite[Proposition 6]{Cla}).
Let $\iota :G \ra B_G Q$ be the fiber inclusion.
Then, we have the long exact sequence of homotopy groups
$$ \cdots \ra \pi_n(G) \otimes \Q \stackrel{\iota_*}\lra \pi_n(B_G Q ) \otimes\Q\lra \pi_n(BQ)\otimes\Q \lra \pi_{n-1}(G) \otimes\Q\ra \cdots \ \ \ \ (\mathrm{exact}). $$
Notice that $B_G Q$ includes the Lie group $G$ as a topological submonoid by definitions,
and $\iota$ is a monoid homomorphism.
The induced map $\iota: H_*(G;\R ) \ra H_*(B_G Q;\R ) $ is injective by Proposition \ref{cohomology}.
An observation of the primitive elements implies the injectivity of $\iota_*: \pi_n(G) \otimes \Q \ra \pi_n(B_G Q ) \otimes\Q$.
Thus, \eqref{g09} is divisible by $\sum_k \mathrm{dim}(H^k (G; \mathbb{R}) ) s^k$.
Hence, dividing \eqref{g09} by the Milnor-Moore theorem on $G$, we have the conclusion \eqref{g99}.
\end{proof}


\begin{exa}\label{corex:}
If $Q$ is $S^{2m}$ and $G=SO(2m+1)$ as in Example \ref{rsx},
we can compute the rational homotopy from the Poincar\'{e} series \eqref{g999}.
We focus only on the cases of $m=1,2,3$, and give a list of $\mathrm{rank}\pi_{k}(BS^{2}) $ as follows.
\begin{center}
\begin{tabular}{|c|c|c|c|c|c|c|c|c|c|c|c|c|c|c|c|c|c|c|c|c|c|c|c|c|c|c|c|c|c|c|c|c|c|c|c|c|c|c|c|c|c|c|c|c|c|c|c|c|}
\hline
$k $& 1 & 2 & 3 &4 &5 & 6 & 7 & 8 & 9 & 10 & 11 & 12 & 13 & 14 & 15 & 16 & 17 \\ \hline
$\mathrm{rank}\pi_{k}(BS^{2}) $ & 1 & 1 & 1& 1& 1 & 2 & 2 & 2 & 3 & 5 & 6 & 7 &11 & 27 &47& 85& 151 \\ \hline
$\mathrm{rank}\pi_{k}(BS^{4}) $ & 1 & 1 & 0& 0& 1 & 1 & 1 & 1& 1 & 2 & 2& 2 &3& 7 &11 & 16& 23 \\ \hline
$\mathrm{rank}\pi_{k}(BS^{6}) $ & 1 & 1 & 0& 0& 0 & 0 & 1 & 1& 1 & 1 & 1& 1&1& 3 &5 & 7& 10 \\ \hline
\end{tabular}
\end{center}
\end{exa}

\section{Continuous $\R$-value rack cocycles.}\label{A33}
In Sections \ref{A33}--\ref{A326}, 
we focus on the rack space $BX^{\delta}$, where $X^{\delta}$ means a smooth quandle with descrete topology.
The cohomology of $BX^{\delta}$ coincides with the rack cohomology \cite{FRS1,FRS2,FRS3}, and has applications to low-dimensional topology; see, e.g., \cite{CJKLS,CKS,IK,Nos3}.

For this, let us briefly review rack cohomology \cite{FRS1,FRS2,FRS3}. 
Let $X$ be a quandle. 
Then, $C_n^R(X)$ is defined to be the free right $\Z$-module generated by $X^n$.
For $( x_1, \dots,x_n) \in X^n$, we define $\partial^R_n (x_1, \dots,x_n) $ by
$$ \sum_{1\leq i \leq n} (-1)^i\bigl( (x_1, \dots,x_{i-1},x_{i+1},\dots,x_n) -( x_1\lhd x_i,\dots,x_{i-1}\lhd x_i,x_{i+1},\dots,x_n) \bigr) \in C_{n-1}^R(X).$$
This yields a homomorphism 
$\partial^R_n : C_n^R(X) \rightarrow C_{n-1}^R(X)$ such that $ \partial^R_n \circ \partial^R_{n+1}=0. $
Dually, for an abelian group $A $, we have the cochain complex $C^n_R(X;A)$ defined by $\Hom (C_n^R(X), A) $
with the dual operation of $\partial^R_n $.
As seen in, e.g., \cite{CJKLS,CKS,IK}, for applications to low-dimensional topology, it is important to concretely describe an $n$-cocycle as a map $X^n \ra A$ with $n \leq 4 $.


In this section, we will restrict on the continuous subcochain group.
Let $Q$ be a smooth quandle of the form $(G,H,z_0)$.
That is, we consider the subcomplex of $C^n_R(Q;\R)$ defined by
$ C^n_{\rm cont}( Q ):= \{ f: Q^n \ra \R \ | \ f \mathrm{ \ is \ continuous} \}$,
which was first studied in \cite{ESZ}, and called {\it the continuous cohomology}.
Furthermore, we introduce a class of $Q$:
\begin{defn}[{cf. homogeneousness in \cite{LN}}]\label{r43}
Fix $m \in \Z$. The smooth quandle $Q$ is said to be {\it semi-homogenous} (of level $m$),
if for any $a \in Q$ there is a zero measure set $O_a$ such that
the $C^{\infty}$-map $ Q \setminus O_a \ra Q \setminus ( a \lhd O_a) $ which sends $x $ to $a \lhd x $ is a covering of degree $m$.
\end{defn}
\begin{exa}\label{r3}
For example, the quandle on the $m$-sphere $S^m$ is semi-homogenous of level 2.
Indeed, letting $q \in S^m $ be the antipodal point against $a$, and
$O_a$ be the equator between $a$ and $q $, we can easily show
the map $ Q \setminus O_a \ra Q \setminus \{ q \} $ is a covering of degree $2$.
In parallel, since the projective spaces $\R P^m , \mathbb{C}P^m$ are quotients of some spheres, we can easily see that $\R P^m $ and $ \mathbb{C}P^m$ are semi-homogenous.

More generally, we conjecture that, if $X$ is the smooth quandle from every compact symmetric space (explained in Example \ref{kkk3}), $X$ may be semi-homogenous.
In fact, T. Nagano \cite{Nagano} introduced the concept of ``centrosome", and he and M. S. Tanaka gave many examples of centrosome, which indicate zero-measure sets $ O_a$ satisfying Definition \ref{r43}.
\end{exa}

\begin{exa}\label{r4}We will consider the case where $Q$ is semi-homogenous and of finite order.
Then, $O_q$ must be the empty set; thus, the covering $ Q \ra Q $ which sends $x$ to $a \lhd x$ must be bijective. Namely $m=1$. This bijectivity was called homogenous property in \cite{LN}.
\end{exa}

We will show a theorem, as a continuous version of \cite[Theorem 1.1]{LN}, which assumes semi-homogeneousness.


\begin{thm}\label{co43e}
If a transitive smooth quandle $Q=G/H$ is semi-homogenous and compact,
every cocycle in $C^n_{\rm cont}( Q ) $ is cohomologous to a constant map.
In particular, the cohomology $H^n_{\rm cont}( Q ) $ is $\R.$
\end{thm}
In conclusion, in order to obtain non-trivial rack cocycles of $Q$,
we should assume neither compactness of $G$ nor the continuous $\R$-value cochain.
For example, the quandle on $Q=\R^2$ with $x \lhd y = 2y-x$ has a non-trivial
continuous 2-cocycle $X^2 \ra \R$: see Corolally \ref{ddcor}. 
On the other hand, if
$Q=\R/\Z=S^1$ is a quandle with $x \lhd y = 2y-x$, then Proposition \ref{dd} implies that the universal 2-cocycle from $C_2^R(Q;\Z) $ is not continuous.

To prove Theorem \ref{co43e}, we need several lemmas. Using the Haar measure of $G$, we can choose a metric $dy$ on $Q$ which is invariant respect to the action of $G$. We may assume $\int_Q dy = 1$.
\begin{lem}[{cf. Lemmas 3.1 and 3.2 in \cite{LN}}]\label{le9} For any $x, w \in Q$ and any continuous function $K : Q \ra \R$, the following equalities hold.
$$\int_QK( x \lhd y) dy = \int_Q K( ( x \lhd w) \lhd y) dy= \int_Q K(( x \lhd y) \lhd w ) dy. $$
\end{lem}
\begin{proof} We begin by computing the first term as 
\[ \int_Q K(x \lhd y) dy = \int_{Q \setminus x \lhd O_x \ } K( x \lhd y) dy = m \int_{Q \setminus O_x } K( y' ) dy' = m \int_Q K(y')dy'.\]
By replacing $x $ by $x \lhd w$, we similarly have $\int_Q K( ( x \lhd w) \lhd y) dy = m \int_Q K(y')dy'$, which implies the first equality.
By the right invariance of $dy$, replacing $y$ to $y \lhd^{-1} w$ implies
$$ \int_Q K(( x \lhd y) \lhd w ) dy = \int_Q K(( x \lhd (y \lhd^{-1} w)) \lhd w ) dy = \int_Q K( ( x \lhd w) \lhd y) dy. $$
This is the second equality, exactly.
\end{proof}
Next, we will prepare some maps.
We introduce two maps $ \partial_n^0 $ and $ \ \partial_n^1 $ from $C^n_{\rm cont}( Q ) $ to $C^{n+1}_{\rm cont}( Q ) $ by setting
\[\partial_i^0 (h) (x_1, \dots, x_{n+1})= h (x_1, \dots, x_{i-1}, x_{i+1 } , \dots ,x_{n+1}) ,\]
\[\partial_i^1 (h) (x_1, \dots, x_{n+1})= h (x_1 \lhd x_i , \dots, x_{i-1} \lhd x_i, x_{i+1 } , \dots , x_{n+1}) .\]
By definition, we should notice $\partial_n^R (h) =\sum_{i=1}^{n+1}( -1)^i (\partial_i^0 (h) -\partial_i^1 (h)) $.
In addition, for $j \leq n $, we define $\phi_n^j: C^n_{\rm cont}( Q ) \ra C^n_{\rm cont}( Q ) $ by
$$ \phi_n^j:(h) (x_1,\dots, x_{n}):= \int_{Q^j} h(x_1 \lhd y_1, \dots, x_j \lhd y_j, x_{j+1}, \dots, x_{n} ) dy_1 \cdots dy_j, $$
$\phi_n^{0} $ by the identity map, and $\phi_n^{n+1} $ by $\phi_n^n$.
Furthermore, we define $D_n^j: C^{n}_{\rm cont}( Q ) \ra C^{n-1}_{\rm cont}( Q ) $ by
$$ D_n^j(k) (x_1,\dots, x_{n-1}):= \int_{Q^j} k(x_1 \lhd y_1, \dots, x_{j-1} \lhd y_{j-1}, x_j , y_{j}, x_{j+1}\dots, x_{n -1 } ) dy_1 \cdots dy_j, $$
and $D_n^{n} $ by the zero map. Here, we should compare \cite{LN}; Precisely,
if $Q$ is of finite order, the maps $\phi_n^j $ and $ D_n^j$ coincide with the maps defined in
\cite[\S 3]{LN}. In addition, we give lemmas as relation among the above maps:
\begin{lem}[{cf. Lemmas 3.3--3.8 in \cite{LN}}]\label{le11} The following equalities hold.
\[ \partial_i^0 \circ D_{n}^j(h) =\partial_i^1 \circ D_{n}^j (h) \ \ \ \ \ \ \ \ \ \ \mathrm{for } \ \ 1 \leq i\leq j \leq n, \]
\[ D_{n+1}^j\circ \partial_i^0(h) = D_{n+1}^j\circ \partial_i^1(h) \ \ \ \ \ \mathrm{for } \ \ 1 \leq i\leq j \leq n ,\]
\[ \partial_{j+1}^0 \circ D_n^j(h) =\phi_n^{j-1} (h) \ \ \ \ \ \ \ \ \ \ \ \ \mathrm{for } \ \ 1 \leq j <n ,\]
\[ \partial_{j+1}^1 \circ D_n^j(h) =\phi_n^{j} (h) \ \ \ \ \ \ \ \ \ \ \ \ \ \ \mathrm{for } \ \ 1 \leq j < n, \]
\[D_{n+1}^j\circ \partial_i^0(h) = \partial_{i+1}^0 \circ D_n^j(h) \ \ \ \ \ \ \mathrm{for } \ \ 1 \leq j< i \leq n +1,\]
\[D_{n+1}^j\circ \partial_i^1(h) = \partial_{i+1}^1 \circ D_n^j(h) \ \ \ \ \ \ \mathrm{for } \ \ 1 \leq j< i \leq n +1.\]
\end{lem}
\begin{proof}
The proofs are almost same as those of Lemmas 3.3--3.8 in \cite{LN}, respectively.
Thus, we only show the first equality. We now denote $a \lhd b$ by $a^b$ for simplicity.
For $i<j$, we can easily show that
$\partial_i^0 \circ D_{n}^j(h) (x_1,\dots, x_n )$ is equal to
$$ \int_{Q^j } h( x_1^{ y_1}, \dots , x_{i-1}^{ y_{i-1}}, x_{i+1}^{y_{i+1}}, \dots, x_{j-1}^{y_{j-1}}, x_j, y_j, x_{j+1}, \dots, x_n )dy_1 \cdots dy_j , $$
and, that $\partial_i^1 \circ D_{n}^j(h)(x_1,\dots, x_n )$ is equal to
$$ \int_{Q^j } h( x_1^{{y_1}^{x_i\lhd y_i }} , \dots , x_{i-1}^{{y_{i-1}}^{x_i\lhd y_i }}, x_{i+1}^{y_{i+1}}, \dots, x_{j-1}^{y_{j-1}}, x_j, y_j, x_{j+1}, \dots, x_n )dy_1 \cdots dy_j . $$
In addition, if $i=j$, we similarly have
\[\partial_j^0 \circ D_{n}^j(h) = \int_{Q^{j} } h( x_1^{ y_1}, \dots , x_{j-1}^{ y_{j-1}}, y_j, x_{j+1} , \dots, x_n )dy_1 \cdots dy_{j} ,\]
\[\partial_j^0 \circ D_{n}^j(h) = \int_{Q^{j} } h( x_1^{ y_1 \lhd x_j }, \dots , x_{j-1}^{ y_{j-1}\lhd x_j }, y_j, x_{j+1} , \dots, x_n )dy_1 \cdots dy_{j}. \]
Applying Lemma \ref{le9} $i-1$ times and Fubini theorem to the integrals,
we obtain the equality $ \partial_i^0 \circ D_{n}^j(h) =\partial_i^1 \circ D_{n}^j (h) $ as required.
\end{proof}
Putting all this together, we have
\begin{prop}[{cf. Proposition 3.1 in \cite{LN}}]\label{le14} For $j>1$, $D_n^j: C^*_{\rm cont}( Q ) \ra C^{*+1}_{\rm cont}( Q ) $ is a chain homotopy from $\phi^j_n$ to $\phi^{j-1}_n.$
\end{prop}
\begin{proof}
The computation in the proof is same as that of Proposition 3.1 in \cite{LN}, by using Lemma \ref{le11}. Thus, we may omit writing the detailed computation.
\end{proof}
\begin{proof}[Proof of Theorem \ref{co43e}]
This proposition implies that every cocycle in $C^n_{\rm cont}( Q ) $ is cohomologous to the map $\phi^n_n (h) $.
By the proof of Lemma \ref{le9}, we notice
$$ \int_{Q^n} h (x_1 \lhd y_1, \dots, x_n \lhd y_n) dy_1 \cdots dy_n= m^n\int_{Q^n} h (y'_1, \dots, y_n') dy'_1 \cdots dy'_n . $$
Namely, this $\phi^n_n (h) $ does not depend on $x_1, \dots, x_n$, that is, a constant map.
To summarize, every cocycle in $C^n_{\rm cont}( Q ) $ is cohomologous to a constant map, as required.
\end{proof}
\section{Rack cocycles from secondary characteristic classes.}\label{A326}
In order to get non-trivial rack cocycles of quandles,
we will introduce an algorithm to obtain $\C/\Z$-value rack cocycles from
the secondary characteristic classes.

Our approach in this section is based on the works of Dupont and Kamber \cite{Dup,Dup2,DK}.
Thus, \S \ref{A36} reviews the works, and \S \ref{A3432} describes the algorithm.

\subsection{Review of Dupont \cite{Dup,Dup2}; presentations of group cocycles}\label{A36}

First, we prepare some homogenous complexes.
Given a set $X$ acted on by a group $G$,
let $C_n^{\Delta}(X) $ be the free $\Z$-module generated by $
(n+1)$-tuples of $X$, that is, $C_n^{\Delta}(X) = \Z\langle X^{n+1 } \rangle $. This $C_n^{\Delta}(X) $ has a differential operator $ \partial_*^{\Delta}$ defined by
$$ \partial_n^{\Delta}(x_0,\dots, x_n) = \sum_{i: \ 0 \leq i \leq n} (-1)^i (x_0,\dots, x_{i-1},x_{i+1}, \dots, x_n). $$
The action of $G$ on $X$ gives rise to the diagonal action on $C_n^{\Delta}(X) $. Denote by $ C_n^{\Delta}(X)_{G} $ the coinvariant $ C_n^{\Delta}(X)\otimes_{\Z[G]} \Z $.
For example, if $X=G$ with natural action of $G$, the complex $ C_*^{\Delta}(G) $ gives a $\Z[G]$-free resolution of the augmentation $\Z[G] \ra \Z$.
Therefore, the homology of $C_*^{\Delta}(G)_{G} $ is isomorphic to the ordinary group homology of $G$.

Next, we will explain Proposition \ref{co4e5} below.
Let $V$ be a manifold which is $(q-1)$-connected for some $q \in \Z$,
and $G$ be a Lie group with transitive action on $V$.
Let $C_{*}^{\rm sing}(V)$ be the chain complex of smooth singular simplicies in $V$.
This chain complex is naturally made into a right $\Z[G]$-module, and is acyclic of length $q-1$.
Then, we can find a chain transformation $\sigma $ of $G$-modules, which ensures the following commutative diagram: 
\begin{equation}\label{datte}
{\normalsize
\xymatrix{ \Z
\ar@{=}[d] & & C_0^{\Delta}(G) \ar[ll]_{\partial_0^{\Delta}} \ar[d]^{\sigma }& &C_1^{\Delta}(G )\ar[ll]_{\partial_1^{\Delta}} \ar[d]^{\sigma } & \ar[l]_{\ \ \ \ \partial_2^{\Delta}} \cdots & \ar[l]_{\!\!\! \partial_q^{\Delta}} \ar[d]^{\sigma }C_q^{\Delta}(G) \\
\Z & & C_{0}^{\rm sing} (V) \ar[ll]_{\partial_0} & & C_{1}^{\rm sing} (V)\ar[ll]_{\partial_1} & \ar[l]_{\ \ \ \ \partial_2} \cdots & \ar[l]_{\!\!\! \partial_q} C_{q}^{\rm sing} (V) .
}}
\end{equation}
As is known as the comparison theorem, this $\sigma$ is unique up to homotopy.
Furthermore, for a $G$-invariant complex value $q$-form $\omega$, we define a cochain $ \mathcal{C}(\omega) \in \Hom (C_q^{\Delta }(G)_G), \C )$ by
\begin{equation}\label{k12}\mathcal{C}(\omega) (g_0, g_1, \dots, g_q ):= \int_{\sigma ( g_0,g_1,\dots, g_q)} \omega, \end{equation}
for $g_0,g_1, \dots, g_q \in G.$ The following is due to Stokes theorem.
\begin{prop}[{\cite[Proposition 10.4]{Dup2}}]\label{co4e5}
Suppose that $\omega$ is closed, and that the integral $\int_z \omega$ lies in $\Z$ for any $z \in C_q^{\rm sing}(V;\Z)$.

Then, the cochain $ \mathcal{C}(\omega) $ is a $q$-cocycle mod $\Z$. Furthermore, it is nullcohomologous if $\omega = d \omega'$ for some $G$-invariant $(q-1)$-form $\omega'.$
\end{prop}
As an insightful result, Dupont-Kamber \cite{DK} showed that this formulation includes Chern-Simons classes as follows:
\begin{exa}[\cite{DK}.]\label{core2}
Let $G$ be $GL_k(\C)$, and $V$ be $GL_k(\C)/ GL_{k-1}( \C) $. By Bott periodicity, $V$ is $(2k-2)$-connected, and has $ H^{2k-1}(V;\Z) \cong \Z$. Since $V$ is the complexification of the compact homogeneous space $U(k)/U(k-1)$, the generator of the $(2k-1)$-th cohomology group can be represented by a complex value $G$-invariant $(2k-1)$-form $\omega_k$.
Then, the group cocycle $\mathcal{C}(\omega_k ) \in H^{2k-1 } ( GL_n(\C); \C/\Z)$ is shown to be equal to the $k$-th Chern-Simons class.
\end{exa}
\subsection{Relation to secondary characteristic classes}\label{A3432}
Under the condition in the previous subsection, we will show that
every secondary characteristic classes in the sense of \cite{Dup2,DK} produces an $n$-cocycle in the rack complex.

For this, we review Inoue-Kabaya map \cite{IK}.
Let $Q$ be a smooth quandle of the form $(G,H,z_0)$.
For $n \in \Z_{n\geq 2} , $ consider the following set composed of maps:
\begin{equation}\label{jiuy222}
I_n := \bigl\{ \ \iota : \{ 2, 3, \dots, n\} \lra \{ 0,1\} \ \bigr\}.
\end{equation}
For a tuple $(x_0,\dots, x_n) \in Q^{n+1 }$ and for each $\iota \in I_n$, we define
$x(\iota, i) \in Q$ by 
$$ x(\iota, i):= ( \cdots ( ( x_i \lhd^{\iota(i+1)} x_{i+1}) \lhd^{\iota(i+2)} x_{i+2}) \cdots ) \lhd ^{\iota(n)} x_n .$$
Here $x \lhd^0 y =y$.
Choose $p \in Q $. If $n\geq 2$, we define a homomorphism
\[ \varphi_n: C_n^R(Q ;\Z ) \lra C_n^{\Delta}(Q)_{G}, \]
by setting
\[ \varphi_n(x_1,\dots,x_n) := \sum_{\iota \in I_n} (-1)^{\iota(2) +\iota (3) + \cdots +\iota(n)}\bigl( p,x(\iota,1),\dots, x(\iota, n)\bigr).\]
If $n=1$, we define $\varphi_1(a)=(p,a).$
This $\varphi_n$ is shown to be a chain map. Namely, $ \partial_n^{\Delta}\circ \varphi_n =\varphi_{n-1} \circ \partial_{n}^{R}$.

Next, we review a $(G,H)$-projectivity of the complex $C_n^{\Delta}(Q) $ from \cite[\S 3]{NM}.
To this aim, an exact sequence $ N \stackrel{i}{ \ra }M \stackrel{j}{ \ra } L $ of right $\Z[G]$-module homomorphisms is {\it $(G,H)$-exact}, if the kernel of $j$ is a direct $\Z[H]$-module
summand of $M$.
A right $\Z[G]$-module $A$ is said to be {\it $(G,H)$-projective} if, for every $(G,H)$-exact sequence $0 \ra N \stackrel{i}{ \ra} M \stackrel{j}{ \ra } L\ra 0 $, and every $\Z[G]$-homomorphism $\psi : A \ra L$, there is a $\Z[G]$-homomorphism $\psi' : A \ra M$ such that $q \circ \psi ' = \psi $. Then, it is shown \cite[Proposition 3.10]{NM} that
the above module $C_{n}^{\Delta}(Q) $ is $(G,H)$-projective, and
the following sequence is $(G,H)$-exact:
$$ \cdots \stackrel{ \partial_{n+1}^{\Delta}}{\lra} C_{n}^{\Delta}(Q) \stackrel{ \partial_n^{\Delta}}{\lra} \cdots \ra C_{1}^{\Delta}(Q) \stackrel{ \partial_{1}^{\Delta}}{\lra} C_{0}^{\Delta}(Q) \lra \Z \lra 0. $$

Moreover, we can easily verify that the bottom sequence in \eqref{datte} is also $(G,H)$-exact.
Thus, by $(G,H )$-projectivity (see \cite[Proposition 3.11]{NM}),
the chain map $\sigma$ factors through a chain $\Z[G]$-map $ \tau: C_n^{\Delta}(Q) \ra C_{n}^{\rm sing} (Q)$ for $n \leq q$.
Here, the choice of $\tau$ is unique up to homotopy.
Hence, similarly, for any $G$-invariant $q$-form $\omega$ such that $\int_z \omega$ lies in $\Z$ for any $z \in C_q^{\rm sing}(Q;\Z)$, it can be easily shown that the following map is a $q$-cocycle modulo $\Z$.
\begin{equation}\label{k15}\mathcal{T}(\omega):Q^{q+1} \lra \C/\Z ; \ \ \ \ \ (x_0, x_1,\dots, x_q )\longmapsto \int_{\tau ( x_0, x_1,\dots, x_q)} \omega . \end{equation}
On the other hand, since $ C_q^{\Delta}(Q) $ is a $\Z[G]$-module,
the above chain map $\tau $ in \eqref{datte} factors through $C_q^{\Delta}(Q) $.
In conclusion, we have
\begin{prop}\label{d4e5}
Let $ \omega $ be the $q$-cocycle satisfying the assumption in Proposition \ref{co4e5}.
Then, the pullback $\varphi_{q}^{*} (\mathcal{T}(\omega) ) \in C_R^q(Q; \C/\Z) $ is a rack $q$-cocycle. \end{prop}
As mentioned in Example \ref{core2}, the class of cocycles presented by $ \mathcal{T}(\omega) $
contains a class of generalized Chern-Simons classes.
In summary, such generalized classes can be represented as rack cocycles. Hence, it is reasonable to hope that this proposition produces many rack cocycles of $X$, when $X$ is a subquandle of $V$.
\begin{exa}\label{opp}
In the paper of Inoue-Kabaya \cite{IK}, they consider the case
$(PSL_2(\C), H,z_0)$, where $H$ is the unipotent subgroup ${\small \Bigl\{ \left(
\begin{array}{cc}
1 & b \\
0 & 1
\end{array}
\right) \Bigr|\ \ b \in \C \Bigr\}} $ and $z_0 = {\small \left(
\begin{array}{cc}
1 & 1 \\
0 & 1
\end{array}
\right) }.$
We remark that $ G/H$ is bijective to $(\mathbb{C} \times \mathbb{C} \setminus \{ (0,0)\})/ \pm $.
In this case, Chern-Simons 3-class $\widehat{C}_3$ is well-understood (see, e.g., \cite[Charters 7--11]{Dup} or \cite[\S 7]{IK}), and is represented by a map $\widehat{C}_3: V^4 \ra \C/4 \pi^2 \Z$ with a cocycle expression. Furthermore, $\widehat{C}_3$ has a close relation to complex volume of hyperbolic 3-manifolds. For this reason, the paper \cite{IK} presented $\widehat{C}_3$ as a rack 3-cocycle, and gave a result on the complex volume.
\end{exa}
\appendix
\section{Proofs of Theorems \ref{thddm1} and \ref{thm2}.}\label{A1}

We will prove Theorems \ref{thddm1} and \ref{thm2}.
The outline of the proofs are based on \cite{Dup,Cla}: precisely, Dupont \cite{Dup} showed a
de Rham theory of simplicial manifolds, and Clauwens \cite{Cla} constructed a triangulation of $\square$-sets, which induced a ring isomorphism on cohomology;
Thus, we give a bridge between their results, and give the proof of the theorems. 

For this purpose, we first prepare notation on simplicial manifolds from \cite{Dup}. {\it A simplicial manifold} $Y$ is defined as a sequence of manifolds $Y_n$ for $n \in \N$ together with face maps $\delta_i : Y_n \ra Y_{n-1}$ for $0 \leq i \leq n$ such that
$$\delta_{j-1} \delta_i =\delta_i\delta_j \mathrm{ \ \ \ for \ any \ \ }0 \leq i< j \leq n.$$
Let $\Delta^p \subset \R^{p+1}$ be the standard simplex
$$\Delta^p := \{ t=(t_0, \dots, t_p) \in \R^{p+1} \ | \ t_i \geq 0, \ \sum_{0 \leq i \leq p} t_i =1 \} ,$$
and let $\epsilon^i: \Delta^{p-1} \ra \Delta^p $ be the $i$-th face map.
Then, the fat realization $| \! | Y| \! |_{\Delta} $ of $Y$ is the quotient space of
$\sqcup_{p \geq 0} \Delta^p \times Y$, with the identifications
$$(\epsilon^i (t) ,y) \sim (t, \delta_i y), \ \ \ \ \ \ t \in \Delta^{p-1}, s \in Y_p, \ i=0,1,\dots, p .$$
Then, we denote $\A^*_{\Delta}( Y_p)$
by the DGA consisting of $n$-forms on $\Delta^p \times Y_p $ which are extended to $C^{\infty}$ forms on $ (\sum_i t_i =1 )\times Y_p$.
Moreover,
an {\it $n$-form} $\varphi$ on $Y$ is
a sequence of $n$-forms $\phi^{(p)} \in \A_{\Delta} ^n( Y_p)$ of $C^{\infty}$-class satisfying $ (\epsilon^{i} \times \id)^* \phi^{(p)}= (\id \times \delta_{i} )^* \phi^{(p-1)} $ for any $i\in \{1, \dots, p\}$. Then, we can define the de Rham cohomology of $\A_{\Delta} ^*( Y_*)$.
Furthermore, we decompose $ A_{\Delta} ^* (Y)$ into a sum $A_{\Delta} ^n (Y)= \bigoplus_{n=k+\ell} \A_{\Delta}^{k,\ell}(Y)$,
where $ A^{k,\ell}_{\Delta}(Y)$ is composed of the forms $\varphi$ of type $(k,\ell )$,
i.e., $\varphi$ restricted to $ \Delta^p \times Y_p$ is $ q_1^* (\phi^{(k)}_I ) \times q_2^* (\phi^{(\ell)}_Y ) $
for some $ \phi^{(k)}_I \in \A_{\Delta}^k (I^p)$ and $ \phi^{(\ell)}_Y \in \A_{\Delta}^{\ell} (Y_p)$. Here $q_{1} : \Delta^p \times Y_p \ra \Delta^p $ and
$q_{2} : \Delta^p \times Y_p \ra Y_p$ are the projections.
Also let $d_{\Delta}$ (resp. $d_Y$) denote the pullback of exterior differential on $\A_{\Delta}^* (\Delta^p)$ (resp. on $\A_{\Delta}^* (Y_p) $).
Thus, we have a double complex $(A^{k,\ell}_{\Delta}(Y), d_{\Delta}, d_Y )$,
and the total complex $(A^*_{\Delta}(Y), d)$, where $d= d_{\Delta}+ d_Y $.
Further, we consider another double complex $(\mathcal{A}_{\Delta}^{k,\ell}(Y), \delta, d_Y)$ where $\delta= \sum_{i=1}^p (-1)^i \delta_i .$




Following \cite[\S 3.2]{Cla}, we give a triangulation from a $\square$-set.
For $n \in \N$, let $[n]$ denote the set $\{1,2,\dots, n\}$.
A $k$-partition of $[n]$ is a sequence $S = (S_1; S_2;\cdots ; S_k)$ of nonempty
subsets of $[n]$ which are mutually disjoint and satisfy $[n]= S_1 \cup \cdots \cup S_k.$

Given a cubical manifold $X$, we define a simplicial manifold $T (X)$, as a manifold analogy of \cite[\S 3]{Cla}.
The set of $k$-simplicies $T (X)_k$ consists of the pairs $(x; S)$, where $x \in X_n$ and $S$ is a $k$-partition of $[n].$
The boundary maps are given by
\begin{equation}\label{k1}\ \delta_0(x; S_1;\cdots ; S_k) = (\delta_{S_1}^1 x; \theta_{S_1} (S_2); \cdots ; \theta_{S_1}(S_k)), \end{equation}
\begin{equation}\label{k2}\delta_i(x; S_1; \cdots ; S_k) = (x; S_1; \cdots ; S_{i-1}; S_i \cup S_{i+1}; S_{i+2}; \cdots ; S_k) \ \ \mathrm{ for} \ \ 0 < i < k, \end{equation}
\begin{equation}\label{k3}\delta_k(x; S_1; \cdots ; S_k) = (\delta^0_{S_k}(x); \theta_{S_k} (S_1); \cdots ; \theta_{S_k} (S_{k-1})). \end{equation}
Here, for $S \subset [n]$ , we write $\theta_S$ for the unique order-preserving map from $[n] - S$ to $[n - \# (S)]$.
Then, it is not so hard to check that $T(X)$ is a simplicial manifold by definitions. 
Although the definition of $T(X)$ seems complicated, here is Figure \ref{sle6} on a triangulation with $k=2$ and $k=3$.

\begin{figure}[h]
\begin{picture}(-30,32)

\put(170,-27){\pc{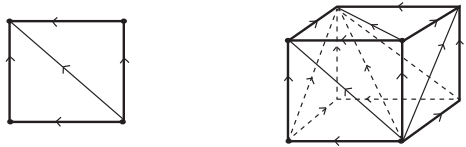}{0.4526}}
\end{picture}
\

\

\caption{\label{sle6} The canonical triangular decompositions of the square and the cube.}
\end{figure}
Next, we will show Lemma \ref{le1}. Given $\mathfrak{t}=(t_1, \dots, t_k) \in [0,1]^k$, we may choose a sequence $S_{\mathfrak{t}} :=(i_1, \dots,i_k) \in \{ 1,\dots, k\}^k $ with $ 1 \geq t_{i_1} \geq\cdots \geq t_{i_k}\geq 0 $ such that $i_1, \dots,i_k$ are mutually distinct. Then, for $x \in X_n$, we correspond $ \Phi ( t_1 ,\dots t_k, x) $ to
$$\bigl( (1 -t_{i_1}, t_{i_1} - t_{i_2} , \dots, t_{i_{k-1}} - t_{i_k} ,t_{i_k} ) , (x ; i_1; \cdots;i_k )\bigr) \in \Delta^k \times T(X)_k $$
Then, we can verify, by \eqref{k1} and \eqref{k3}, that the correspondence descends to a continuous map $\Phi: | \!| X | \!| \ra| \!| T (X) | \!|_{\Delta} $ on geometric realizations. Furthermore,
\begin{lem}\label{le1} For any cubical manifold $X$, the map $\Phi: | \!| X | \!| \ra | \!| T (X) | \!|_{\Delta} $ is a homeomorphism.
\end{lem}
\begin{proof} To construct the inverse mapping $\Psi $, we prepare notations.
Suppose $(x; S_1;\dots; S_k) $ with $x \in X_n$ and $n\geq k$.
We take the composite map
$$\mu_{ S_1;\dots; S_k} := \epsilon^{|S_1|+|S_2| + \cdots + |S_{k}| }\circ\cdots \circ \epsilon^{|S_1|+|S_2| } \circ \epsilon^{ |S_1|}: \Delta^k \lra \Delta^n. $$
Decompose $ (S_1;\dots; S_k) \subset [n] $ as $(s_1, \dots, s_n) \in \mathbb{N}^n$ pointwise.
Furthermore, we regard this $(s_1, \dots, s_n)$ as a permutation $\sigma \in \mathfrak{S}_n$ and set up another map defined by
$$ \Upsilon : \Delta^n \lra I^n; \ ( t_0' , \dots, t_n') \longmapsto (t_1'+\cdots + t_{n}' ,t_2'+\cdots + t_{n}' ,\dots, t'_{n-1} +t'_{n}, t_n'). $$
Denote by $P_{n,k}$ the set of $k$-partitions of $[n]$ with discrete topology.
Then, we define a map $\Psi: \Delta^k \times X^n \times P_{n,k} \ra I^n \times X^n $ by
$$ \Psi (t_0' , \dots, t_k, x; S_1;\dots; S_k) := ( \Upsilon \circ \sigma^{-1} \circ \mu_{ S_1;\dots; S_k} ( t_0' , \dots, t_k') ,x ) \in I^n \times X^n$$
Then, by \eqref{k1}--\eqref{k3}, this $\Psi$ defines a continuous map $| \!| T (X) | \!|_{\Delta} \ra | \!| X | \!|$. Moreover, it is not hard to check that $\Psi \circ \Phi$ and $\Phi \circ \Psi$ are identities by construction. This completes the proof.
\end{proof}
Following the proof, we can define the pullback $\Psi^*( \phi) \in A^{k,\ell} (T(X)) $ of any form $\phi \in A^{k,\ell} (X) $. Moreover, we can similarly verify that
\begin{lem}\label{le2} The maps $ \Phi ^* : A^{*,*}_{\Delta} (T(X)) \ra A^{*,*}_{\Delta} (X) $ and $ \A^{*,*}_{\Delta} (T(X)) \ra \A^{*,*} (X)$ are bigraded ring isomorphisms.
Here, the inverse mappings are constructed from the pullback $\Psi^*$.
\end{lem}
We now use the above lemmas to prove the theorems.
\begin{proof}[Proof of Theorem \ref{thddm1}]
For any simplicial manifold $Y$, Dupont \cite[Theorem 2.3]{Dup} constructed a chain map
$\mathcal{T}: A^{k,\ell}_{\Delta} (Y) \ra \A^{k,\ell}_{\Delta} (Y) $ which gives a homotopy equivalence.
Hence, when $Y=T(X)$, the composite $ \Phi ^* \circ \mathcal{T} \circ \Psi ^* : A^{*,*} (X) \ra \A^{*,*} (X) $ plays role of a desired chain-map.
\end{proof}

\begin{proof}[Proof of Theorem \ref{thm2}]
Dupont \cite[Theorem 2.14]{Dup} considered the map in the $E_{\infty}$-term induced from $\mathcal{T}: A^{k,\ell}_{\Delta} (Y) \ra \A^{k,\ell}_{\Delta} (Y) $, and
the induced map $ A^{*}_{\Delta} (Y) \ra A^{*}_{\Delta} (|\!| Y|\!| )$ is 
multiplicative. The above maps $\Phi$ and $\Psi$ are multiplicative by definitions.
Thus, the map in the $E_{\infty}$-term induced from $ \Phi ^* \circ \mathcal{T} \circ \Psi^*$ is also multiplicative. This completes the proof.
\end{proof}

\section{Some computation of quandle homology of smooth quandles.}\label{B1}

In this section, we will compute rack quandle homology of ``linear" quandles.
Fix $\omega \in \R\setminus \{0,1 \} $ and $n \in \N$.
Let us assume that $ X$ is either a quandle on $\R^n $ with $x \lhd y= \omega x + (1- \omega )y$ or a quandle on $(\R/ \Z) ^n$ with $x \lhd y= 2y- x$. (cf.
the classification of smooth homogenous manifolds of dimension $\leq 2$; see Ishikawa \cite[\S 6]{Ishi}).
\begin{prop}\label{dd}
If $\omega \neq \pm 1$ and $\omega \in \Q$, $H_2^R (X;\Z)$ is $\Z$.
However, if $ \omega =- 1$, $H_2^R (X;\Z)$ is isomorphic to $(\R^n \wedge_{\Q} \R^n) \oplus \Z $.

If $X= (\R/ \Z) ^n$ with $x \lhd y= 2y- x$, then $H_2^R (X;\Z)$ is isomorphic to $(\R/\Q)^n \wedge_{\Q} (\R/\Q)^n\oplus \Z$. 
\end{prop}
\noindent
The key for the proof is the result of Clauwens \cite{Cla2}.
Precisely, 
the paper computed the rack homology from the isomorphism
\begin{equation}\label{z1} H_2^R (X;\Z) \cong \Z \oplus \frac{X \otimes_{\Z} X }{\{ x\otimes y - \omega y \otimes x\}_{x,y \in X}}; \ \ \ \ \ \ \ n (a,b) \longmapsto (n , [ (a-b) \otimes b]).
\end{equation}
\begin{proof}
We will compute the right hand side in details.
Recall elementary computations
\begin{equation}\label{z2} \Q/\Z \otimes \Q/\Z =0, \ \ \Q \otimes_{\Z} \Q \cong \Q, \ \ \ \mathrm{ and} \ \ \ \R \otimes_{\Z} \R \cong \R \otimes_{\Q} \R.
\end{equation} 
Hence, if $X=\R^n$ with $\omega \neq \pm 1 $, one has $H_2^R (X;\Z) \cong \Z $, because $ x \otimes y = \omega x \otimes \omega y= \omega^2 x \otimes y$ in \eqref{z1}.
On the other hand, if $\omega =-1 $, the right hand side of \eqref{z1} turns out to be the exterior product as stated above.

Finally, we consider $X= (\R/ \Z) ^n$ with $x \lhd y= 2y- x$.
Notice $\R /\Z \cong \Q/\Z \oplus (\bigoplus_{\lambda} \Q)$ as a $\Z$-module, where $\lambda$ runs over an uncountable index set. Thus, $\R /\Q \cong\bigoplus_{\lambda} \Q$.
Thus, the computation of $H_2^R (X;\Z) $ immediately follows from \eqref{z1} and \eqref{z2}.
\end{proof}
\begin{cor}\label{ddcor}
Let $Q=\R^2$ be the quandle with $x\lhd y =2y-x$.
Then, the map $\mathcal{C}: Q^2 \ra \R$ which takes $((x_1,y_1),(x_2,y_2)) $ to $x_1y_2 - x_2 y_1$ is a continuous 2-cocycle and is not null-cohomologous.
\end{cor}
\begin{proof} Consider the $\Q$-linear map $q : \R^2 \wedge_{\Q} \R^2 \ra \R$ which takes $(x,y) \wedge (z,w)$ to $ xw-yz$. According to \eqref{z1}, the map $\mathcal{C}': Q^2 \ra \R^2 \wedge_{\Q} \R^2 $ which sends $(a,b) $ to $ (a-b ) \wedge b$ gives a universal 2-cocycle. Thus, the composite $ q \circ \mathcal{C}'$ is not null-cohomologous.
Noticing $ \mathcal{C}= q \circ \mathcal{C}'$ completes the proof.
\end{proof}

Moreover, corresponding Example \ref{opp}, for a field $F$, we mention the second cohomology of $X_F$,
where $G=PSL_2(F), \ H= \bigl\{ \left(
\begin{array}{cc}
1 & a \\
0 & 1
\end{array}
\right) \bigr\} _{a \in F }$ and $ z_0 = \left(
\begin{array}{cc}
1 & 1 \\
0 & 1
\end{array}
\right)$.
Moreover, we recall the Milnor $K_2$-group $K_2(F)$ which is isomorphic to $F^{\times } \otimes_{\Z} F^{\times } / \{ a \otimes (1-a)\}_{a \in F \setminus \{ 0,1\}} $.
If $F=\C$, $K_2(F)$ is known to be uniquely divisible, i.e., a direct sum of $\Q$'s,
\begin{prop}[{A special result of \cite[Corollary 8.5]{Nos3} }]\label{dd44} If $F=\C $, then
$H_2^R (X_F ;\Z) \cong \Z \oplus \C \oplus K_2(\C ) $.

Furthermore, if $F=\R $, then
$H_2^R (X_F ;\Z) \cong \Z \oplus \Z \oplus \R \oplus K_2(\C)^{+} $, where $K_2(\C)^{+} $ is the
invariant part of $K_2(\C ) $ with respect to the conjugate operation $\bar{}: \C \ra \C$.
\end{prop}
Concerning quandles on the spheres, W. E. Clark and M. Saito \cite{CS} studied some phenomena of quandle 2-cocycles, together with a relation to knot invariants.


\vskip 1pc

\normalsize

DEPARTMENT OF
MATHEMATICS
TOKYO
INSTITUTE OF
TECHNOLOGY
2-12-1
OOKAYAMA
, MEGURO-KU TOKYO
152-8551 JAPAN

\end{document}